\documentclass[12pt]{article}

\usepackage{amssymb,amsthm,amsfonts,amsmath,amscd,verbatim}
\usepackage{xy}
\usepackage[hypertex]{hyperref}

\date{}
\newtheorem{tr}{Theorem}
\newtheorem{lemma}[tr]{Lemma}
\theoremstyle{definition}
\newtheorem{df}[tr]{Definition}

\def\differential{d}
\renewcommand\d\differential

\DeclareMathOperator\Ker{Ker}

\DeclareMathOperator\im{Im}

\def\smat#1{\left(\begin{smallmatrix}#1\end{smallmatrix}\right)}

\def\P{\bb P}

\def\k{\Bbbk}
\renewcommand\Im\im

\def\bb#1{\mathbb #1}
\def\cal#1{\mathcal #1}
\def\dd#1#2{\frac{\partial #1}{\partial#2} }

\def\defeq{:=}

\let\star *
\let\ast *
\let\subset\subseteq

\let\NUM ¹

\def\defeq{:=}

\def\W2{W_2(\k)}
\def\Wn{W_n(\k)}
\def\CW2{\A_{W_2}}
\def\A{\cal A}
\def\AW{\A_{W_2}}

\title{On the annihilators of rational functions in the Lie algebra of derivations of $\k[x, y]$.}
\author{O.~G.~Iena
\footnote{Johannes Gutenberg-Universit\"at Mainz, e-mail:
iena@mathematik.uni-mainz.de}, A.~P.~Petravchuk \footnote{Kiev
Taras Shevchenko University, e-mail: aptr@univ.kiev.ua,
}, A.~O.~Regeta
\footnote{Kiev Taras Shevchenko University, e-mail:
andriyregeta@gmail.com}}

\begin{document}

\maketitle
\begin{abstract}
Let $\k$ be an algebraically closed field of zero characteristic.
The Lie algebra $W_2=\W2$ of all $\k$-derivations of the
polynomial ring
 $\k[x, y]$ naturally acts on
the polynomial ring $\k[x, y]$ and also on the field of rational
functions $\k(x, y)$. For a fixed rational function $u\in \k(x,
y)\setminus \k$ we consider the set $\A_{W_2}(u)$ of all
derivations $D\in W_2$ such that $D(u)=0$. We prove that
$\A_{W_2}(u)$ is a free submodule of rank $1$ of the $\k[x,
y]$-module $W_2$. A description of the maximal abelian subalgebras
as well of the centralizers of elements in the Lie algebra
$\A_{W_2}(u)$ has been obtained.
\end{abstract}

\section*{Introduction}

Let $\k$ be a field of characteristic zero.  The Lie algebra
$W_n=\Wn$ of all $\k$-derivations of the polynomial ring
$\k[x_{1}, \ldots , x_{n}]$ was studied by many authors from
different points of view.  Subalgebras of $W_{n}$  that are free
$\k[x_{1}, \ldots , x_{n}]$-submodules of maximal rank in $W_n$,
were studied by V.~M. Buchstaber and D.~V.~Leykin in \cite{Buch}.
Using results of D.~Jordan~\cite{Jordan1} one can point out some
classes of simple subalgebras of $W_n$ that are also $\k[x_{1},
\ldots , x_{n}]$-submodules of the $\k[x_{1}, \ldots ,
x_{n}]$-module $W_n.$
 In ~\cite{Petien} the centralizers of elements and the maximal
abelian subalgebras of the algebra $sa_2(\k)$ of all derivations
 $D\in W_2$ with zero divergence have been studied.

 In this paper
 we study a class of subalgebras of the Lie algebra $\W2$ over an algebraically closed
  field of characteristic zero. This class is determined by the natural
 action of the  Lie algebra $\W2$  on the field of rational
functions $\k(x, y)$. Recall that every derivation $D\in \W2$ of
the ring $\k[x, y]$ can be uniquely extended to a derivation of
the field $\k(x, y)$. It is natural to consider for a fixed
rational function $u\in \k(x, y)\setminus \k$ the set $\AW(u)$ of
all derivations $D\in W_2$ such that $D(u)=0$. This set will be
called the annihilator of $u$ in $\W2$. It is  a Lie subalgebra of
$\W2$ and at the same time a $\k[x, y]$-submodule of the $\k[x,
y]$-module $\W2$.

Using some results from~\cite{Bodin}, \cite{Now}, ~\cite{Oll},
\cite{Petien2} we prove (Theorem~\ref{tr:free}) that for a
rational function $u\in \k(x, y)\setminus \k$
 its annihilator $\AW(u)$ is a free submodule of rank $1$ in the $\k[x, y]$-module $W_2$.
  We give also a free generator of this module. We describe the centralizers
  of elements and the maximal abelian subalgebras
  of the Lie algebra $\AW(u)$ (Theorems~\ref{tr:centralizers poly} and~\ref{tr:centralizers rational}).
It turned out that the algebra $\AW(u)$ has completely different
structure in the cases when $u$ is a polynomial and when $u$ is a
rational function of the form $u=p/q$ with algebraically
independent polynomials $p$ and $q$.

The notations used in the paper are standard. For a rational
function $u\in \k(x, y)$ we denote by $\tilde u$ its generative
 rational function, i.~e., a generator of the maximal subfield
 in $\k(x, y)$ of transcendence degree $1$ that contains $u$.
Recall (see~\cite{Petien2} for details) that $\tilde u$ is
defined uniquely up to  linear fractional transformations
\[
\frac{\alpha \tilde u+\beta}{\gamma \tilde u+\delta},\quad
\alpha\delta-\beta\gamma\neq 0.
\]
Note also that if $u$ is a polynomial there exists a polynomial
generative function $\tilde u\in \k[x, y]$. Recall that if a rational
 function or a polynomial is generative for itself, then it is called closed.

A derivation $D=P\dd{}{x}+Q\dd{}{y}$ will be called reduced if the
polynomials $P$ and $Q$ are coprime, i.e. $\gcd (P, Q)=1$. For an
arbitrary polynomial $u\in \k[x, u]$ we denote by $D_u$ the
derivation of
 $\k(x, y)$ given by the rule $D_u(\varphi)=\det J(u,
  \varphi)=\dd{u}{x}\dd{\varphi}{y}-\dd{u}{y}\dd{\varphi}{x}$, i.~e.,
\[
D_u=-\dd{u}{y}\dd{}{x}+\dd{u}{x}\dd{}{y}.
\]

If $u$ possesses a polynomial generative function,
then one can choose $\tilde u$ to be an irreducible polynomial.
If $u$ does not have a polynomial generative function, then one
 can choose $\tilde u = p/q$ for some irreducible and
 algebraically independent polynomials $p$ and $q$ (see, for example, \cite{Bodin} or
  \cite{Petien2}, Corollary 1).

\section{On the structure of the $\k[x, y]$-module $\A_{W_2}(u)$.}

\begin{lemma}\label{lemma:integrating factor}
Let $D=P\dd{}{x}+Q\dd{}{y}\in \W2$ be a reduced derivation. Then $D$ has a non-trivial kernel, i.~e., $\Ker
D\supsetneq \k$, if and only if there exist non-zero polynomials
$h, u\in \k[x, y]$, $u\not \in \k$, such that $hD=D_u$.
\end{lemma}
\begin{proof}
If $hD=D_u$, then $hD(u)=D_u(u)=\det J(u, u)=0$. As $h$ is different
from zero, one concludes that $u$ belongs to the kernel of $D$.

Let now $D(u)=0$ for some non-constant polynomial $u$. The latter means
 $P\dd{u}{x}+Q\dd{u}{y}=0$ and using that $P$ and $Q$ are coprime
 we obtain $\dd{u}{x}=hQ$ and $\dd{u}{y}=-hP$ for some polynomial $h$. Thus $hD=D_u$.
\end{proof}
\begin{lemma}\label{lemma:centralizer as generating}
Let $u\in \k(x,y) \setminus \k$ and let $\tilde u$ be its
generative rational function. Then $\AW(u)=\AW(\tilde u)$.
\end{lemma}
\begin{proof}
Since $\tilde u$ is a generative rational function for $u$, we
 obtain $u=F(\tilde u)$ for some non-constant rational function
 $F(t)\in \k(t)$. Then for every derivation $D\in W_2$ one has
\[
D(u)=D(F(\tilde u))=F'(\tilde u)\cdot D(\tilde u)
\]
and using that $F'(\tilde u)\neq 0$ we conclude that $D(u)=0$ if
and only if $D(\tilde u)=0$. This implies $\AW(u)=\AW(\tilde u)$
and completes the proof.
\end{proof}

Following \cite{Oll} we assign to every irreducible polynomial
 $p\in \k[x, y]$ and every closed rational function $p/q$
 with algebraically independent irreducible $p$ and $q$ reduced derivations $\delta_p$ and
 $\delta_{p, q}$ respectively.  For an irreducible polynomial
 $p\in \k[x, y]$ the derivation
$D_p$ may be written as
$D_p=-\dd{p}{y}\dd{}{x}+\dd{p}{x}\dd{}{y}$. Let $h=\gcd(\dd{p}{x},
\dd{p}{y})$, put $P=-\dd{p}{y}/h$, $Q=\dd{p}{x}/h$, and denote
$\delta_p=P\dd{}{x}+Q\dd{}{y}$.
 Note that $\gcd(P, Q)=1$ and the derivation $\delta_p$ is
 defined by the polynomial $p$ uniquely up to multiplication by a non-zero element from $\k$.

Analogously for a rational function $\varphi=p/q$ such that $p$
and $q$ are irreducible and algebraically independent polynomials
we denote by $D_{p, q}$ the derivation defined from the formula
\[
D_{p, q}(f)\cdot \d x\wedge \d y=(q\, \d p-p\, \d q)\wedge \d f.
\]
One easily computes
\begin{multline*}
(q\, \d p-p\, \d q)\wedge \d f=q\cdot\d p\wedge \d f - p\cdot \d q\wedge \d f=  \\
= \left[ q \left( \dd{p}{x}\dd{f}{y}-\dd{p}{y}\dd{f}{x} \right)
-p \left( \dd{q}{x}\dd{f}{y}-\dd{q}{y}\dd{f}{x} \right) \right] \d
x\wedge \d y,
\end{multline*}
hence we obtain the equality
\[
D_{p, q}=\left(p\dd{q}{y}-q\dd{p}{y}\right)\dd{}{x}+
\left(q\dd{p}{x}-p\dd{q}{x}\right)\dd{}{y}.
\]
Let $P_0=p\dd{q}{y}-q\dd{p}{y}$, $Q_0=q\dd{p}{x}-p\dd{q}{x}$.
Denote $h=\gcd(P_0, Q_0)$ and put
\[
\delta_{p, q}=\frac{1}{h}\cdot D_{p, q}.
\]
Note again that $\delta_{p, q}$ is defined uniquely up to
multiplication by a non-zero constant.
\begin{lemma}\label{lemma:Dpq}
(1) Let $p$ be an irreducible polynomial. Then $D_p(F(p))=0$ for
every rational function $F\in \k(t)$.

\noindent (2)  Let $p$ and $q$ be irreducible algebraically
independent polynomials. Then
\[
D_{p, q}(p)=\det(J(p, q))\cdot p,\quad D_{p,
q}(q)=\det(J(p, q))\cdot q.
\]

\noindent (3)  For every homogeneous polynomial $f(x, y)$ of
degree $m$ it holds
\[
D_{p, q}(f(p, q))=m\det J(p, q)f(p, q).
\]
\end{lemma}
\begin{proof}
(1) As $D_p(p)=0$, we conclude that $D_p(F(p))=F'(p)D_p(p)=0$.

 (2) One computes
\begin{align*}
D_{p, q}(p)=\left(p\dd{q}{y}-q\dd{p}{y}\right)\dd{p}{x}+
\left(q\dd{p}{x}-p\dd{q}{x}\right)\dd{p}{y}=\\
\left(\dd{p}{x}\dd{q}{y}-\dd{p}{y}\dd{q}{x}\right)p=
\det(J(p, q))\cdot p.
\end{align*}
Analogous straightforward computation shows that $D_{p, q}(q)=\det(J(p, q))\cdot q$.

(3)  Let $f(x,y)=\sum\limits_{i=0}^m a_ix^iy^{m-i}$, $a_i\in \k$,
be a homogeneous polynomial of degree $m$ in variables $x$ and
$y$. Then
\begin{multline*}
D_{p, q}(f(p, q))=D_{p, q}\left(\sum\limits_{i=0}^m a_ip^iq^{m-i}\right)=
\sum\limits_{i=0}^m D_{p, q}(a_ip^iq^{m-i})=\\
\sum\limits_{i=0}^m a_i(ip^{i-1}D_{p, q}(p)q^{m-i}+(m-i)q^{m-i-1}D_{p, q}(q)p^i )=\\
 \sum\limits_{i=0}^m a_i (ip^{i}q^{m-i}+(m-i)q^{m-i}p^i )\det(J(p, q))=
 m\det(J(p,q))f(p, q).
\end{multline*}
This proves the last part of the lemma.
\end{proof}
For convenience let us introduce the following notations.
Let $\varphi\in \k(x, y)\setminus \k$ be a non-constant rational function.

If $\varphi$ has a polynomial generative functions,
 then there exists an irreducible generative polynomial
  $p$ of $\varphi$. Put $\delta_\varphi\defeq \delta_p$.

If $\varphi$ does not have any polynomial generative function,
 we find a generative rational function of the form $p/q$
 with irreducible and algebraically independent
 polynomials $p, q\in \k[x, y]$. Put in this case $\delta_\varphi\defeq \delta_{p, q}$.
\begin{lemma}\label{lemma:zero kernel}
Let $D$ be a derivation of \ $\k[x, y]$ and $p, q\in\k[x, y]$ be
two algebraically independent polynomials such that $D(p)=D(q)=0$.
Then $D=0$.
\end{lemma}
\begin{proof}
We can consider $D$ as a derivation of the field $\k(x, y)$.
 Its kernel is an algebraically closed subfield of $\k(x, y)$ by  Lemma 2.1
 from~\cite{Now}. Since $p$ and $q$ are algebraically independent,
 one concludes that $\Ker D=\k(x, y)$. Thus $D=0$. This completes the proof.
\end{proof}
\begin{lemma}\label{lemma:reduced}
Let $D_{1},  D_{2}\in W_{2}$ and let $D_{1}$ be a reduced derivation.
If $uD_{1}+vD_{2}=0$ for some polynomials $u, v\in \k[x, y]$, then
$v$ divides $u$ and
$D_{2}=fD_{1}$ for $f=u/v\in \k[x, y]$.
\end{lemma}
\begin{proof}
Let $D_{1}=P_{1}\dd{}{x}+Q_{1}\dd{}{y}$,
$D_{2}=P_{2}\dd{}{x}+Q_{2}\dd{}{y}$. Then $uP_{1}+vP_{2}=0$ and
$uQ_{1}+vQ_{2}=0$. From these equalities it follows that $v$ divides $uP_1$ and $uQ_1$.
Since the polynomials $P_{1}$ and $Q_{1}$
are coprime, $u$ is divisible by $v$ and we obtain $P_{2}=fQ_{1}, \ Q_{2}=fQ_{1}$ for
$f=u/v$. Hence $D_{2}=fD_{1}$.
\end{proof}
\begin{tr}\label{tr:free}
For an arbitrary rational function $\varphi\in \k(x, y)\setminus
\k$ its annihilator $\AW(\varphi)$ is a free submodule of rank $1$
of the  $\k[x, y]$-module $\W2$. As a free generator of the
submodule $\AW(\varphi)$ one can choose the derivation
$\delta_\varphi$.
\end{tr}
\begin{proof}
Let $\delta_\varphi=P_0\dd{}{x}+Q_0\dd{}{y}$. Then the polynomials $P_0$
and $Q_0$ are coprime by construction.
Let us take an arbitrary  derivation $D=P\dd{}{x}+Q\dd{}{y}$ from $\AW(\varphi)$.
We shall show that $D=h\delta_\varphi$ for some polynomial $h\in \k[x, y]$.

Consider the case when $\varphi$ possesses a polynomial generative function $p$.
Then by Lemma~\ref{lemma:centralizer as generating} we get $\AW(\varphi)=\AW(p)$.
 Therefore, $D(p)=0$ and by Lemma~\ref{lemma:integrating factor} we conclude
  that there exists a polynomial $h_0$ such that $\dd{p}{x}=h_0Q$
  and $\dd{p}{y}=-h_0P$. This means that $D_{p}=h_{0}D.$ By definition of $\delta_p$ there is a polynomial
  $h_{1}\in\k[x, y]$ such that $D_p=h_{1}\delta_p$ and therefore
  $h_{0}D-h_{1}\delta _{p}=0.$ Since $\delta _{p}$ is a reduced
  derivation, we have by Lemma \ref{lemma:reduced} that $D=h\delta _{p}$
  for some polynomial $h\in \k[x, y].$

Let us consider now the case when $\varphi$ does not have any
polynomial generative function. In this case
$\delta_\varphi=\delta_{p, q}$ for some irreducible and
algebraically independent polynomials $p$ and $q$ such that
$p/q$ is a generative rational function for $\varphi$.
By Lemma~\ref{lemma:centralizer as generating} $\AW(\varphi)=\AW(p/q)$.
Since $D_{p, q}(p/q)=0$, from the definition of $\delta_{p, q}$ it
follows that $\delta_{p, q}(p/q)=0$. Therefore,
$\delta_{p, q}$ belongs to $\AW(p/q)=\AW(\varphi)$.

Let $D=P\dd{}{x}+Q\dd{}{y}$ be an arbitrary non-zero derivation
from $\AW(\varphi)$.  Since  $D(\varphi)=0$ implies $D(p/q)=0$ and
since $D(p/q)=\frac{D(p)q-pD(q)}{q^2}$, we conclude
$D(p)q-pD(q)=0$. As the polynomials  $p$ and $q$ are coprime, we
obtain that $D(p)=\lambda p$ and $D(q)=\lambda q$ for some
$\lambda\in \k[x, y]$. Denote for convenience $\mu =\det J(p, q).$
Then by Lemma  \ref{lemma:Dpq} $h\delta _{p, q}(p)=D_{p, q}(p)=\mu
p$ and $h\delta _{p, q}(q)=D_{p, q}(q)=\mu q$. As the polynomials
$p$ and $q$ lie in the kernel of the derivation $\lambda h\delta
_{p, q}-\mu D$ we have by Lemma \ref{lemma:zero kernel} that
$\lambda h\delta _{p, q}-\mu D=0$. The derivation $\delta _{p, q}$
is reduced by construction, so by Lemma~\ref{lemma:reduced} we
obtain $D=h\delta _{p, q}$ for some polynomial $h.$
 We
proved that every derivation $D\in \AW(\varphi)$ is of the form
$h\delta_{\varphi}$ for some polynomial $h\in \k[x, y]$.
 This shows that $\AW(\varphi)$ is a free $\k[x, y]$-module
 of rank $1$ with the generator $\delta_{\varphi}$.
\end{proof}
\section{On centralizers of elements and maximal abeli\-an subalgebras of the Lie algebra $\AW(\varphi)$.}
\begin{df}
Let $p(x, y)$ be an irreducible polynomial. A polynomial $f(x, y)$
will be called $p$-free if $f(x, y)$ is not divisible by any
polynomial in $p(x, y)$ of positive degree.
\end{df}

It is clear that for every polynomial $g(x, y)$ there exists a
$p$-free polynomial $\bar g(x, y)$ such that $g=\bar g\cdot h(p)$
for some $h\in \k[t]$.  Note that $\bar g$ is determined by the
polynomial $g$ uniquely up to multiplication by a non-zero
constant.

\begin{lemma}\label{lemma:centralizers pfree}
Let $u\in \k(x, y)\setminus \k$ be a non-constant rational
function with  a polynomial generative function and $p(x, y)$ be
an irreducible  generative polynomial for $u(x, y)$.
 Then $C_{\AW(u)}(g\delta_p)=C_{\AW(u)}(\bar g \delta_p)$ for any $g\in
\k[x, y]$, where $\bar g$ is the $p$-free polynomial corresponding
to $g.$
\end{lemma}
\begin{proof}
Let $g={\bar g}h$, where $h\in \k[p]$. Take an arbitrary
derivation $D$ from $C_{\AW(u)}(g\delta_p)$. By
Theorem~\ref{tr:free} $D$ is of the form $D=f\delta _{p}$ for some
polynomial $f$. Since $[f\delta _{p}, g\delta _{p}]=0$, we have
$$0=[f\delta _{p}, g\delta _{p}]=[f\delta _{p}, h{\bar g}\delta
_{p}]=f\delta _{p}(h){\bar g}\delta _{p}+h[f\delta _{p}, {\bar
g}\delta _{p}].$$ Since $\delta _{p}(h)=0$ and $h\not= 0$, we
obtain from the last equalities that $D=f\delta _{p}\in
C_{\AW(u)}(\bar g \delta_p)$. Conversely, let $D \in
C_{\AW(u)}({\bar g}\delta_p)$. Using the same notations, write now
$D=f\delta _{p}$ for some $f\in \k[x, y]$. Then $[f\delta _{p},
{\bar g}\delta _{p}]=0.$ But then $$[f\delta _{p},  g\delta
_{p}]=[f\delta _{p}, {\bar g}h\delta _{p}]=f\delta _{p}(h){\bar
g}\delta _{p}+h[f\delta _{p}, {\bar g}\delta _{p}]=0.$$ Therefore,
the derivation $f\delta_p$ belongs to $C_{\AW(u)}(g\delta_p)$ if
and only if it belongs to $C_{\AW(u)}(\bar g\delta_p)$. This
proves the required statement.
\end{proof}
\begin{tr}\label{tr:centralizers poly}
Let $u\in \k(x, y)\setminus \k$ be a non-constant rational
function that
 possesses a polynomial generative function. Let $p(x, y)$ be an irreducible
 generative polynomial for $u(x, y)$. Then

(1) the centralizer of an arbitrary element $f\delta_{p}$ from
$\AW(u)$ in the Lie algebra $\AW(u)$ equals $\bar f
\k[p]\delta_p$, where $\bar f$ is a $p$-free polynomial
corresponding to $f$;

(2) maximal abelian subalgebras of $\AW(u)$ and only they are of
the form $\bar f \k[p]\delta_p$, where $\bar f$ is a $p$-free
polynomial.
\end{tr}
\begin{proof}
(1) Since by Lemma~\ref{lemma:centralizers pfree} \(
C_{\AW(u)}(f\delta_p)=C_{\AW(u)}(\bar f \delta_p), \) we can
assume without loss of generality that $f=\bar f$.
 Take an arbitrary element
$g\delta_p$ from $C_{\AW(u)}(f\delta_p)$. Denote by $\bar g$ a
$p$-free polynomial corresponding to $g$. Then  $g=\bar g\cdot
h_0$, where $h_0=h_0(p)$ is a polynomial in $p$. By Lemma
\ref{lemma:centralizers pfree} it holds
 $[f\delta _{p},
{\bar g}\delta _{p}]=0$ and therefore $\delta _{p}(f)\bar
g-f\delta _{p}(\bar g)=0$. This relation yields the equality
$\delta _{p}(\bar g/f)=0$. As $D_{p}=\lambda \delta _{p}$ for some
polynomial $\lambda$, we have $D_{p}(\bar g/f)=0$. Since $D_p(\bar
g/f)=\det(J(p,\bar g/f))$, the latter implies that the rational
functions $p$ and $\bar g/f$ are algebraically dependent (see for
example~\cite{Hodge}, Ch.~III, \S~7, Th.~III or \cite{Petien2},
Lemma~1).
 As the
polynomial $p$  is closed (it is irreducible), the function $\bar
g/f$ belongs to the field $\k(p)$. This means that the rational
function $\bar g/f$ can be written in the form $\bar
g/f=u(p)/v(p)$ for some coprime polynomials $u, v\in \k[t]$. From
the last equality it follows $\bar gv(p)=fu(p)$. As the
polynomials $f$ and $\bar g$ are both $p$-free, we have ${\bar
g}=cf$ for some $c\in \k^{\times}$ (a $p$-free polynomial
corresponding to the polynomial $\bar gv(p)=fu(p)$ is determined
uniquely up to nonzero constant multiplier). Hence $g=h_{0}{\bar
g}=h_{0}cf$, where $h_{0}c\in \k[p].$
 We proved the inclusion $C_{\AW(u)}( f \delta_p)\subset
\k[p]\cdot f \delta_p$.

Since for every polynomial $r\in \k[t]$ we have
\[
[r(p) f\delta_p,   f \delta_p]= (r(p) f \delta_p( f)-
f\delta_p(r(p) f))\delta_p= (r(p) f \delta_p( f)- f r(p)\delta_p(
f))\delta_p=0,
\]
it holds also $\k[p]\cdot  f \delta_p\subset C_{\AW(u)}( f
\delta_p)$. We proved the equality $\k[p]\cdot  f \delta_p =
C_{\AW(u)}( f \delta_p)$ for $p$-free $f$. This proves the
first part of the theorem.

(2) Let $M$ be a maximal abelian subalgebra in $\AW(u)$ and let
$f\delta_p$ be an arbitrary non-zero element of $M$. Then
$M\subset C_{\AW(u)}(f\delta_p)$ and  by the part (1) of this
theorem $C_{\AW(u)}(f\delta_p)=\k[p]\cdot \bar f\delta_p$. Since
for arbitrary polynomials $F, G\in \k[t]$ it holds
\begin{align*}
[F(p)\bar f\delta_p, G(p)\bar f\delta_p]=
&(F(p)\bar f\delta_p(G(p)\bar f)-G(p)\bar f\delta_p(F(p)\bar f))\cdot\delta_p\\
=&(F(p)G(p)\bar f\delta_p(\bar f)-F(p)G(p)\bar f\delta_p(\bar f))\cdot\delta_p=0,
\end{align*}
one sees that  $\k[p]\cdot \bar f\delta_p$ is an abelian algebra.
The maximality of $M$ implies $M=\k[p]\cdot \bar f\delta_p$.
Conversely, the subalgebras of the form $\k[p]\cdot \bar
f\delta_p$ are abelian Lie algebras for any $p$-free polynomial
$\bar f$. As every element commuting with ${\bar f}\delta_p$
belongs by definition to the centralizer $C_{\AW(u)}({\bar
f}\delta_p)=\k[p]{\bar f}\delta $, one sees that all such
subalgebras are maximal abelian. This completes the proof of the
theorem.
\end{proof}
\begin{df}
Let $p$ and $q$ be algebraically independent irreducible
polynomials from the ring $\k[x, y]$.  A polynomial $f(x, y)\in
\k[x, y]$ will be called $p$-$q$-free if $f$ is not divisible by
any homogeneous polynomial in $p$ and $q$ of positive degree.
\end{df}
It is clear that for every polynomial $f$ there exists a
$p$-$q$-free polynomial $\bar f$ such that $f=\bar f h$ for some
homogeneous  in $p$ and $q$ polynomial $h$. We denote by $\k[p,
q]_{m}$ the vector space of all polynomials $f(p, q)$,  where
$f(x, y)$ is a homogeneous polynomial of degree $m$.
\begin{lemma}\label{lemma:different factors}
Let $p, q\in \k[x, y]$ be irreducible algebraically independent
polynomials, $\alpha _{1}, \alpha _{2}, \beta _{1}, \beta _{2}\in
\k$.  Then the polynomials $\alpha_1p+\beta_1q$ and
$\alpha_2p+\beta_2 q$ are either linearly dependent (over $\k$) or
coprime. In the first case $(\alpha_1:\beta_1)=(\alpha_2:\beta_2)$
as points in $\P_1(\k)$. In the second case
$(\alpha_1:\beta_1)\neq(\alpha_2:\beta_2)$.
\end{lemma}
\begin{proof}
Since $p$ and $q$ are algebraically independent, they are also
linearly independent. If the  polynomials $r_1=\alpha_1p+\beta_1q$
and $r_2=\alpha_2p+\beta_2 q$ are
 linearly dependent then clearly $(\alpha_1:\beta_1)=(\alpha_2:\beta_2)$.
Let   now $r_{1}$ and $r_{2}$ be linearly independent. Then
$\det\smat{\alpha_1&\beta_1\\\alpha_2&\beta_2}\neq 0$. We can
write down $p=a_1r_1+b_1r_2$ and $q=a_2r_{1} +b_2 r_{2}$ for some
$a_1, a_2, b_1, b_2\in \k$. It is clear that any common divisor of
$r_1$ and $r_2$ is  a common divisor of $p$ and $q$. Since $p$ and
$q$ are coprime, we conclude that  $r_1$ and $r_2$ are coprime as
well. In this case $(\alpha_1:\beta_1)\not= (\alpha_2:\beta_2)$.
\end{proof}
\begin{tr}\label{tr:centralizers rational}
Let $\varphi$ be a rational function that does not possess a polynomial generative
function. Let $p$ and $q$ be algebraically independent and irreducible polynomials
such that $p/q$ is a generative rational function for $\varphi$. Then
\par
(1) for an arbitrary element $f\delta_{p, q}\in \AW(\varphi)$ its
centralizer in $\AW(\varphi)$ coincides with $\k[p, q]_m\cdot \bar f \delta_{p,
q}$, where $\bar f$
 is a $p$-$q$-free polynomial such that $f=\bar f \cdot h$ for some polynomial $h$ homogeneous of degree $m$ in
 $p$ and $q$;
\par
(2) every maximal abelian subalgebra from $\AW(\varphi)$ is of the
form $\k[p, q]_m\bar f\delta_{p, q}$ for some integer $m$ and
$p$-$q$-free polynomial $\bar f$. Every subalgebra of the type
$\k[p, q]_m\bar f\delta_{p, q}$ is maximal abelian. In particular
all maximal abelian subalgebras of $\AW(\varphi)$ are finite
dimensional.
\end{tr}
\begin{proof}
First of all note that one can choose for $\varphi$  a generative
rational function of the form $p/q$ with irreducible and
algebraically independent polynomials
 $p$ and $q$ by Corollary~1 from~\cite{Petien2}.

(1) By Lemma~\ref{lemma:centralizer as generating}
$\AW(\varphi)=\AW(p/q)$.
 By Theorem~\ref{tr:free} every derivation from $\AW(\varphi)$ may be written
 as $f\delta_{p, q}$ for some polynomial $f\in k[x, y]$.

Take an arbitrary element $g\delta_{p, q}\in
C_{\AW(\varphi)}(f\delta_{p, q})$. Then $[g\delta_{p, q},
f\delta_{p, q}]=0$ and hence $g\delta_{p, q}(f)-f\delta_{p,
q}(g)=0$. Therefore, $\delta_{p, q}(f/g)=0$ and also $D_{p,
q}(f/g)=0$ (recall that $D_{p, q}=\lambda \delta _{p, q}$ for some
$\lambda \in \k[x, y]$). Since $D_{p, q}(f/g)=\det(J(p/q,f/g))$,
the latter implies that the rational functions $p/q$ and $f/g$ are
algebraically dependent (see for example~\cite{Hodge}, Ch.III, \S
7, Th.~III or \cite{Petien2}, Lemma~1). As $p/q$ is a closed
rational function, we see that $f/g\in \k(p/q)$.  Therefore,
\[
\frac{f}{g}=\frac{F(\frac{p}{q})}{G(\frac{p}{q})}
\]
for some coprime polynomials $F, G\in \k[t]$.

By our assumption the  field $\k$ is algebraically closed, so we
can decompose $F$ and $G$ into linear factors, say
\[
F(t)=c_1(t-\lambda_1)\dots (t-\lambda_k),\quad
G(t)=c_2(t-\mu_1)\dots(t-\mu_l),\quad c_1, c_2, \lambda_i,\mu_j\in \k.
\]
Then
\[
\frac{f}{g}=\frac{F(\frac{p}{q})}{G(\frac{p}{q})}=
\frac{c_1(p-\lambda_1q)\dots (p-\lambda_kq)}
{c_2(p-\mu_1q)\dots(p-\mu_lq)}\cdot q^{l-k}
\]
and hence
\[
g=  \frac{c_2(p-\mu_1q)\dots(p-\mu_lq)}{c_1(p-\lambda_1q)\dots (p-\lambda_kq)}\cdot q^{k-l}\cdot f.
\]
Write $f=\bar f \cdot h$ for a $p$-$q$-free polynomial $\bar f$
and a  polynomial $h$ homogeneous of degree $m$ in $p$ and $q$. It
is known that the polynomial  $h$ can be decomposed into the
product
\[
h=(\alpha_1p-\beta_1 q)\dots(\alpha_m p-\beta_m q).
\]
with some $\alpha _{i}, \beta _{j}\in \k$. As $f=\bar fh$, we can
finally write
\[
g=  \frac{c_2(p-\mu_1q)\dots(p-\mu_lq)(\alpha_1p-\beta_1 q)\dots(\alpha_m p-\beta_m q)}
{c_1(p-\lambda_1q)\dots (p-\lambda_kq)}\cdot q^{k-l}\cdot \bar f.
\]
The polynomial  $\bar f$ is not divisible by homogeneous
nonconstant  polynomials in $p$ and $q$, so using
Lemma~\ref{lemma:different factors} we conclude that the rational
function
\[
h_1=\frac{c_2(p-\mu_1q)\dots(p-\mu_lq)(\alpha_1p-\beta_1 q)\dots(\alpha_m p-\beta_m q)}
{c_1(p-\lambda_1q)\dots (p-\lambda_kq)}\cdot q^{k-l}
\]
must be a polynomial, i.~e., all factors of its denominator must
occur as factors in the numerator. It is obvious  that $h_1$ is a
homogeneous polynomial of degree $l+m+(k-l)-k=m$ in $p$ and $q.$

We proved that $g=h_1\cdot \bar f$,
hence $g\delta_{p, q}=h_1\cdot \bar f\delta_{p, q}\in \k[p, q]_m\cdot \bar f\delta_{p, q}$.
Therefore, $C_{\AW(\varphi)}(f\delta_{p, q})\subset \k[p, q]_m\cdot \bar f\delta_{p, q}$.

For every polynomial $h_2$ homogeneous   of degree $m$ in $p$ and
$q$   applying Lemma~\ref{lemma:Dpq}, 3), one obtains
\begin{align*}
[h_2\bar f\delta_{p, q}, \  f\delta_{p, q}]&= [h_2\bar f\delta_{p,
q}, \ h{\bar f}\delta_{p, q}]=({\bar f}\delta _{p,
q}(h_{2})h-h_{2}{\bar f}\delta _{p, q}(h)){\bar f}\delta _{p, q}=\\
&=(m\det J(p, q)h_{2}h-h_{2}m\det J(p, q)h){\bar f}^{2}\delta _{p,
q}=0.
\end{align*}

Therefore, $\k[p, q]_m\cdot \bar f\delta_{p, q}\subset C_{\AW(\varphi)}(f\delta_{p, q})$
and we obtain the equality
\[
C_{\AW(\varphi)}(f\delta_{p, q})=\k[p, q]_m\cdot \bar f\delta_{p, q}.
\]
This proves the first part of the theorem.

(2) This part can be proved  similarly to the  part 2) of
Theorem~\ref{tr:centralizers poly} by replacing the set $\k[p]$ by
$\k[p, q]_m$.
\end{proof}

\section*{Acknowledgements}
The research of the second author was partially supported by DFFD,
Grant F28.1/026.

\def\cprime{$'$} \def\cprime{$'$}


\end{document}